\documentclass[hidelinks,onefignum,onetabnum]{siamart250211}
\usepackage{amsmath,amssymb}

\graphicspath{{Figure/}}

\makeatletter
\def\EquationsBySection{\def\theequation
	{\thesection.\arabic{equation}}%
	\@addtoreset{equation}{section}}

\makeatother \EquationsBySection

\allowdisplaybreaks[3]
	
\title{Spurious-mode-free finite element method for scattering resonances in transmission problems}
\date{}
\author{Bo Gong \thanks{School of Mathematics, Statistics and Mechanics, Beijing University of Technology, Beijing, 100124, China. ({\tt  gongbo@bjut.edu.cn}).}
\and Jiguang Sun \thanks{Department of Mathematical Sciences, Michigan Technological University, Houghton, MI 49931, U.S.A. ({\tt  jiguangs@mtu.edu}).}}

\begin{document}
\maketitle

\begin{abstract}
    Scattering resonances arise in wave phenomena and play an important role in many applications. While extensive theoretical studies have been conducted, effective numerical computation remains limited, and most existing methods suffer from spurious modes. In this paper, we propose a spurious-mode-free method for computing scattering resonances in transmission problems. The unbounded domain is truncated using a Dirichlet-to-Neumann (DtN) map. The resonances are formulated as eigenvalues of a holomorphic Fredholm operator function, which is discretized by the finite element method. The spectrum indicator method is then used to compute the eigenvalues of the nonlinear matrix eigenvalue problems. We establish optimal order convergence and present extensive examples that demonstrate the effectiveness of the proposed method. The results are consistent with existing theoretical findings in the literature and offer new insights that may inform further theoretical developments.
\end{abstract}


\section{Introduction}
Scattering resonances play an important role in wave applications such as acoustics, electromagnetics, and quantum mechanics. For example, resonances are associated with unusual physical phenomena that are critical in the design of novel materials. Mathematically, they are defined as the poles of the meromorphic continuation of the scattering operator \cite{LaxPhillips1989, dyatlov2019mathematical}. In this paper, we do not distinguish between poles and resonances, although the terminology may vary in different contexts. Resonances encode important physical information about the scatterer \cite{Uberall1983}. Extensive theoretical investigations, such as the distribution, bounds on location, and asymptotic behavior of resonances, have been carried out (see, e.g., \cite{Popov1999}).
In contrast, computational methods for scattering resonances are far from mature. 

Existing numerical approaches can be divided into two main categories. The first category consists of methods based on boundary integral equations (BIEs), which automatically satisfy the outgoing condition \cite{SteinbachUnger2012, Misawa2017SIAM, steinbach2017combined, hiptmair2022spurious, grubivsic2023detecting, ma2023computation, LiuSunZhang2025, Matsushima2025ProcA}. These methods require discretization only on the boundary, leading to relatively small algebraic systems. For instance, a Galerkin boundary element method is analyzed in \cite{SteinbachUnger2012}. A combined integral equation approach is proposed in \cite{steinbach2017combined}, and Nyström methods have been employed in \cite{ma2023computation} to achieve high accuracy. However, BIE-based methods often suffer from non-physical (spurious) resonances, as the spectrum of a BIE formulation is generally not equivalent to that of the corresponding PDE problem. This issue is particularly severe for transmission problems, where both spurious and physical resonances lie in the fourth quadrant of $\mathbb{C}$, making them difficult to distinguish.

The second category consists of finite element methods (FEM), which typically truncate the unbounded domain. The perfectly matched layer (PML) technique has been employed in \cite{kim2009computation, NannenWess2018BIT, Halla2022SINUM}. While PML formulations may lead to linear eigenvalue problems, they can introduce spurious eigenvalues that are difficult to identify \cite{steinbach2017combined, NannenWess2018BIT}. Other approaches, such as Hardy space infinite elements, have also been investigated \cite{HohageNannen2009}. All of the aforementioned methods exhibit spurious resonances to some extent, despite attempts to mitigate this issue through mode shifting or parameter tuning \cite{Misawa2017SIAM, NannenWess2018BIT}. Recently, a spurious-mode-free finite element method using the Dirichelt-to-Neumann mapping was proposed to compute scattering resonances of impenetrable obstacles \cite{xi2024analysisfiniteelementdtn}. Using the abstract approximation theory for eigenvalue problems of holomorphic Fredholm operator functions \cite{karma1996approximation, karma1996approximation2}, the error estimates were obtained. The resulting nonlinear matrix eigenvalue problems are solved by the spectral indicator method (SIM) based on contour integrals \cite{xi2023parallel}. 

Compared to impenetrable obstacles, scattering resonances of penetrable obstacles are more complicated and interesting. According to \cite{Popov1999}, if $\Omega$ is smooth and convex, and index of refraction is greater than $1$, there exists an infinite sequence of resonances tending to the real axis from below. These resonances are closely related to the quasi-resonances or resonance frequencies \cite{Lenoir1992, grubivsic2023detecting}. In contrast, if the refractive index inside is less than $1$, there exists a strip free of resonance beneath the real axis \cite{galkowski2018quantum, moiola2019acoustic, dyatlov2019mathematical}. In this paper, we employ the finite element DtN method to the transmission problem. The optimal order of convergence for eigenvalues is proved. The impact of the order of the truncated DtN mapping is studied numerically. Various examples are presented to demonstrate the effectiveness of the proposed method and the theoretical results in the literature. The results on non-convex/non-smooth domains also provide new insights that may inform further theoretical developments.

Scattering resonances can be formulated as the eigenvalues of holomorphic Fredholm operator functions. The computational approach of FEM+DtN+SIM has been applied to compute the scattering resonances of several problems including acoustic obstacles \cite{xi2024analysisfiniteelementdtn}, fluid-solid interaction \cite{XiJi2025}, Schr\"{o}dinger operators \cite{Gong2026AML}, metallic grating structures \cite{xi2024finite}.
The rest of the paper is organized as follows. In Section 2, we introduce the scattering resonances in transmission problems, reduce the problem to a bounded domain using the DtN mapping, and formulate resonances as eigenvalues of an operator function. In Section 3, we analyze the truncated DtN mapping, propose a finite element method, and prove the optimal convergence of the resonances. Abstract approximation theory for holomorphic Fredholm operator functions is briefly recalled. In Section 4, extensive numerical experiments are presented to demonstrate the convergence order, the impact of the truncation order for the DtN mapping, and impact of the radius of the truncated domain. In particular, the results confirm the theory of the behavior of resonances and the resonance-free zone on the complex plane for transmission problems. Section 5 contains conclusions and future work.




\section{Scattering Resonances in Transmission Problems}
Let $\Omega\subset \mathbb{R}^2$ be a bounded Lipschitz domain. Denote by $\Omega^c = \mathbb{R}^2 \setminus \overline{\Omega}$ and $\Gamma = \partial\Omega$. 
The refractive index is given by
\begin{equation*}
n(x) = \left\{
\begin{aligned}
&n_i(x),\quad & &x\in \Omega,
\\
&1,\quad & &x\in \Omega^c,
\end{aligned}
\right.
\end{equation*}
with $n_i(x)> 1$ or $0<n_i(x) < 1$. Let  $k \in \mathbb C$ be the wavenumber. Denote by $u^{i}$ the incident field, $u$ the scattered field, and $u^{t}$ the total field. The transmission problem is to find $u^t$ (or $u$) such that
\begin{align}
\Delta u^{t} + k^2 n(x) u^{t} &= 0,\quad \text{in} \; \mathbb{R}^2,
\label{u^t1}
\\
u^{t} &= u^{i} + u,
\label{u^t2}
\end{align}
where $u$ satisfies the Sommerfeld radiation condition
\begin{equation}
\frac{\partial u}{\partial r} - iku = o(r^{-1/2}),\quad r = |x|\rightarrow \infty.
\label{u^t3}
\end{equation}
Equivalently, the above problem can be written as the problem of finding $u$ such that
\begin{align}
\Delta u + k^2 n(x) u &= g,\quad \text{in}\;\mathbb{R}^2
\label{u1}
\\
\frac{\partial u}{\partial r} - iku &= o(r^{-1/2}),\quad r \rightarrow \infty,
\label{u2}
\end{align}
where $g = k^2(1-n(x)) u^{i}$. The well-posedness of the above problem, see, e.g., Lemma~2.2 of \cite{moiola2019acoustic}, holds when $\text{Im}(k) \geqslant 0$ and $k\neq 0$, i.e., \eqref{u1}--\eqref{u2} has a unique solution $u \in H^1_{\text{loc}}(\mathbb{R}^2)$. The scattering (solution) operator $S(k)$ is such that $ u = S(k)g$ for $g \in L^2(\Omega)$. $S(k)$ can be meromorphically continued to  the lower half complex plane. The poles of $S(k)$ are called scattering resonances \cite{zworski1999resonances, taylor2010partial}. In fact, the Sommerfeld radiation condition no longer holds when $k$ is such that ${\rm Im}(k)<0$ and outgoing condition should be used for the scattered field. In particular, $u$ is said to be outgoing if it has the series expansion of the form
\begin{equation}\label{outgoingconditon}
u(x)=\sum_{n=-\infty}^{\infty}a_nH_n^{(1)}(kr)e^{in\theta},\ \ \ |x|>r_0,
\end{equation}
where $\theta=\arg(x)$, $r_0>0$ is some constant large enough, and $H_n^{(1)}(\cdot)$ is the first kind Hankel function of order $n$.

The transmission problem \eqref{u1}-\eqref{u2} is defined on $\mathbb{R}^2$. We truncate the domain and apply the DtN mapping. Let $\Omega_R$ be a disk centered at the origin with radius $R$ and denote its boundary by $\Gamma_R$. Here $R$ is large enough such that $\Omega_R$ contains $\Omega$ in its interior. The DtN mapping $T(k)$ is given by (see, e.g., \cite{hsiao2011error})
\begin{equation}
T(k)u = \sum_{n=0}^{+\infty}{\!}^{{}^{\scriptstyle{\prime}}} \frac{k}{\pi}\frac{{H^{(1)}_n}^{\prime}\!\!(kR)}{H^{(1)}_n(kR)}\int_{0}^{2\pi} u(R,\phi)\cos(n(\theta - \phi))\text{d}\phi.
\label{T}
\end{equation}
The prime on the summation denotes that the zeroth term has weight one-half. For simplicity, we introduce
\begin{equation*}
a_n(u) = \frac{1}{\pi}\int_0^{2\pi}u(R,\phi)\cos(n\phi)\text{d}\phi,\quad
b_n(u) = \frac{1}{\pi}\int_0^{2\pi}u(R,\phi)\sin(n\phi)\text{d}\phi.
\end{equation*}
Then 
\begin{equation*}
T(k)u = \sum_{n=0}^{+\infty}{\!}^{{}^{\scriptstyle{\prime}}}\frac{k{H^{(1)}_n}^{\prime}(kR)}{H^{(1)}_n(kR)}( a_n(u) \cos(n\theta)
+ b_n(u)\sin(n\theta)).
\end{equation*}
Denote by $Z$ the set of all the zeros of Hankel function $H^{(1)}_n(\cdot)$ for all natural number $n$, and by $\mathbb{R}^-$ the set of non-positive real number. Define $\Lambda = \mathbb{C}\setminus(\mathbb{R}^- \cup Z)$.

Using the DtN mapping, we obtain a problem on the bounded domain $\Omega_R$ that is equivalent to \eqref{u1}-\eqref{u2}:
\begin{align}
\Delta u + k^2 n(x) u &= g\quad \text{in\;}\Omega_R,
\label{uR1}
\\
\partial_{\nu} u &= T(k) u\quad \text{on\;} \Gamma_R.
\label{uR2}
\end{align}
Let $(\cdot,\cdot)$, $(\cdot,\cdot)_1$ and $\langle\cdot,\cdot\rangle$ denote the $L^2(\Omega_R)$ inner product, the $H^1(\Omega_R)$ inner product, and the $H^{-1/2}(\Gamma_R)$--$H^{1/2}(\Gamma_R)$ duality, respectively. Write $V:= H^1(\Omega_R)$. The weak formulation for \eqref{uR1} and \eqref{uR2} is to find $u \in V$ such that
\begin{equation}
(\nabla u,\nabla v) - (k^2n(x)u,v) - \langle T(k)u,v\rangle = (g,v)
\quad \forall 
v\in V.
\label{u weak}
\end{equation}
For $k\in \Lambda$, we define an operator $B(k) : V \rightarrow V$ such that
\begin{equation}
(B(k)u,v)_1 = (\nabla u,\nabla v) - (k^2n(x)u,v) - \langle T(k)u,v\rangle \quad \forall v\in V.
\label{B}
\end{equation}
The scattering resonances are the eigenvalues of $B(\cdot)$,
that is, $k$'s such that
\begin{equation}
B(k)u = 0
\label{Bu=0}
\end{equation}
for some nontrivial $u$. The resonances are the eigenvalues of $B(\cdot)$.

\section{Finite Element Method and Error Estimates}

We first truncate the infinite series for $T(k)$ in \eqref{T} to obtain 
\begin{equation*}
T^N(k)u = \sum_{n=0}^{N}{\!}^{{}^{\scriptstyle{\prime}}}\frac{k{H^{(1)}_n}^{\prime}(kR)}{H^{(1)}_n(kR)}( a_n(u) \cos(n\theta)
+ b_n(u)\sin(n\theta))
\end{equation*}
and define $B^N(k):V\rightarrow V$ such that
\[
(B^N(k)u,v)_1 = (\nabla u,\nabla v) - (k^2n(x)u,v) - \langle T^N(k)u,v\rangle \quad \forall v\in V.
\]

Then we construct a series of conforming finite element spaces $V_n := V_{h_n}$ such that $V_n\subset V$, where $h_n$ denotes the mesh size. The discrete operator $B^N_n(k) : V_n\rightarrow V_n$ is defined by
\[
(B^N_n(k)u_n,v_n)_1 = (\nabla u_n,\nabla v_n) - (k^2n(x)u_n,v_n) - \langle T^N(k)u_n,v_n\rangle \quad \forall v_n\in V_n.
\]
This leads to a nonlinear matrix function $B^N_n(k)$, whose eigenvalues are computed by the parallel SIM \cite{xi2023parallel} based on contour integrals.

The rest of this section is devoted to error estimates for the eigenvalues using the abstract spectral approximation theory of holomorphic Fredholm operator functions \cite{karma1996approximation, karma1996approximation2}. We first prove the convergence of eigenvalues of $B^N(k)$ using the property of the truncated DtN mapping $T^N(k)$. Then we show the optimal order of convergence for eigenvalues of $B^N_n(k)$. 

\subsection{Holomorphic Fredholm Operator Functions}
We recall some basics about holomorphic Fredholm operator functions \cite{karma1996approximation}. Let $X$ and $Y$ be Banach spaces. An operator $F\in \mathcal{L}(X,Y)$ is Fredholm if the dimension of the nullspace $\mathcal{N}(F)$ is finite, the range $\mathcal{R}(F)$ is closed, and the codimension of $\mathcal{R}(F)$ is finite. The index of a Fredholm operator $F$ is defined as $\text{dim}\mathcal{N}(F) - \text{codim}\mathcal{R}(F)$. We say that an operator function $F(\cdot) : \Lambda\rightarrow \mathcal{L}(X,Y)$ is a Fredholm operator function (with index zero), if $F(k)$ is Fredholm (with index zero) for all $k\in \Lambda$.

Let $F(\cdot) : \Lambda\rightarrow \mathcal{L}(X,Y)$ be a holomorphic Fredholm operator function. Define
\begin{equation*}
\rho(F) = \{k\in \Lambda\,:\, F(k)^{-1}\in \mathcal{L}(Y,X)\}\quad\text{and}\quad
\sigma(F) = \Lambda\setminus \rho(F)
\end{equation*}
as the resolvent set and the spectrum of $F(\cdot)$, respectively. If $\rho(F)$ is non-empty, then $\sigma(F)$ has no cluster point in $\Lambda$ and contains only eigenvalues of $F(\cdot)$. Moreover, the inverse $F^{-1}(\cdot):\rho(F)\rightarrow\mathcal{L}(Y,X)$ is holomorphic except poles of finite order.

A vector $(x_0,x_1,\dots,x_j)$ with $x_0\neq 0$ is called a Jordan chain 
of length $j+1$ of $F(\cdot)$ at $\lambda_0\in \sigma(F)$ if
\begin{equation*}
F(\lambda_0)x_{\ell} + F^{(1)}(\lambda_0)x_{\ell-1} + \cdots + \frac{F^{(\ell)}(\lambda_0)}{\ell !}x_0 = 0,\quad \ell = 0, 1,\dots, j.
\end{equation*}
The elements of the Jordan chain are called generalized eigenfunctions,
and the span of all the generalized eigenfunctions of $F(\cdot)$ at $\lambda_0$
is called the generalized eigenspace of $F(\cdot)$ at $\lambda_0$,
denoted by $G(F,\lambda_0)$.
The maximal length of Jordan chains of $F(\cdot)$ at $\lambda_0$
is equal to the order of pole of $F^{-1}(\cdot)$ at $\lambda_0$.

\subsection{Spectral Approximation Theory}
We present the main results of the spectral approximation of eigenvalues of holomorphic Fredholm operator functions \cite{vauinikko1976funktionalanalysis, karma1996approximation, karma1996approximation2}. Let $X$, $Y$, $X_n$, $Y_n$ be Banach spaces and let the operators $p_n\in\mathcal{L}(X,X_n)$ and $q_n\in \mathcal{L}(Y,Y_n)$ satisfy $\Vert p_n x\Vert \rightarrow \Vert x\Vert$ and $\Vert q_n y\Vert \rightarrow \Vert y\Vert$ as $n\rightarrow \infty$ for all $x\in X$ and $y\in Y$. Let $\{x_n\}$ be such that $x_n \in X_n$. If for every subsequence $\{x'_n\}$ of $\{x_n\}$ there exists a subsequence $\{x''_n\}$ of $\{x'_n\}$ and $x\in X$ such that 
$\Vert x''_n - p_n x\Vert \rightarrow 0$ as $n\rightarrow \infty$, then $\{x_n\}$ is called (discrete) compact. The compactness of $\{y_n\}$ is defined analogously.

Let $F\in \mathcal{L}(X,Y)$ and $F_n\in \mathcal{L}(X_n,Y_n)$.
\begin{itemize}
\item
$F_n$ is said to converge to $F$, if $\Vert F_n p_n x - q_n F x\Vert \rightarrow 0,\; n\rightarrow\infty,\; \forall x\in X$.
\item
$F_n$ is said to converge uniformly to $F$, if $\Vert F_n p_n - q_n F\Vert\rightarrow 0$, $n\rightarrow \infty$.
\item
$F_n$ is said to converge stably to $F$, if $F_n$ converges to $F$ and there exists $n_0$ such that $F_n$ is invertible with $\Vert F_n^{-1}\Vert\leqslant C$, $\forall n\geqslant n_0$.
\item
$F_n$ is said to converge compactly to $F$, if $F_n$ converges to $F$ and $\{F_n x_n\}$ is compact for any $\Vert x_n\Vert \leqslant 1$, $\forall n$.
\item
$F_n$ is said to converge regularly to $F$, if $F_n$ converges to $F$ and $\{x_n\}$ is compact for any $\Vert x_n\Vert\leqslant 1$, $\forall n$ such that
$\{F_n x_n\}$ is compact. 
\end{itemize}

We consider an operator function $F(\cdot)$ and a sequence of approximation operator functions $F_n(\cdot)$'s. In particular, we assume that
$F(\cdot) : \Lambda\rightarrow \mathcal{L}(X,Y)$
is a holomorphic Fredholm operator function with $\rho(F)$ non-empty,
and $F_n(\cdot) : \Lambda\rightarrow \mathcal{L}(X_n,Y_n)$
are holomorphic Fredholm operator functions with index zero such that
on each compact $\Lambda_1\subset \Lambda$ it holds 
$
\Vert F_n(k) \Vert \leqslant C,\; \forall k\in \Lambda_1,\; \forall n.
$

Let $\Lambda_0$ be a compact subset 
of $\Lambda$ so that $\partial\Lambda_0\subset
\rho(F)$ and $\Lambda_0\cap \sigma(F) = \{\lambda_0\}$. Denote by $\nu(\Lambda_0,F)$ (or $\nu(\Lambda_0,F_n)$) the sum of algebraic multiplicities of all the eigenvalues
of $F(\cdot)$ (or $F_n(\cdot)$) in $\Lambda_0$.
Then by Theorems 2 and 3 of \cite{karma1996approximation}
and Theorem~2 of \cite{karma1996approximation2},
the following result holds.
\begin{theorem}\label{Thm1}
If $F_n(k)$ converges regularly to $F(k)$ for all $k\in \Lambda$,
then for sufficiently large $n$ and for each $\lambda_n\in \sigma(F_n)\cap\Lambda_0$,
it holds that $\nu(\Lambda_0,F_n) = \nu(\Lambda_0,F)$ and
\begin{equation*}
|\lambda_0 - \lambda_n|\leqslant C\sup_{\substack {x\in G(F,\lambda_0)\\
\Vert x\Vert = 1,\; k\in \partial\Lambda_0}}\Vert q_n F(k)x - F_n(k)p_n x\Vert^{1/\kappa}.
\end{equation*}
Here $\kappa$ is the order of the pole $\lambda_0$ of $F^{-1}(\cdot)$.
\end{theorem}
Let $E, H\in \mathcal{L}(X,Y)$ and $E_n, H_n \in \mathcal{L}(X_n,Y_n)$.
By Theorem 2.55 of \cite{vauinikko1976funktionalanalysis},
if $E_n$ converges stably to an invertible $E$ and $H_n$ converges compactly
to a compact $H$, then $E_n+H_n$ converges regularly to $E+H$.
In addition, if $H_n$ converges uniformly to a compact $H$,
then $H_n$ converges compactly to $H$ 
(see Lemma 3.4 of \cite{xi2024analysisfiniteelementdtn}).
Hence we have the following corollary.
\begin{corollary}\label{EH}
If for each $k\in \Lambda$, $F(k) = E(k) + H(k)$ and $F_n(k) = E_n(k) + H_n(k)$
such that $E(k)$ is invertible, $H(k)$ is compact, $E_n(k)$ converges stably to
$E(k)$ and $H_n(k)$ converges uniformly to $H(k)$, then the conclusion of
Theorem~\ref{Thm1} holds.
\end{corollary}

We further assume that there exist $r_n\in \mathcal{L}(X_n,X)$ and $q'_n\in \mathcal{L}(Y,Y_n)$  such that
\begin{align}
&\Vert r_n\Vert \leqslant C,\; \forall n. \label{rnC}
\\
&\Vert r_np_n x - x\Vert\rightarrow 0,\; n\rightarrow \infty, \; \forall x\in G(F,\lambda_0). \label{rnpnx}
\\
&\Vert q'_n y - q_n y\Vert\rightarrow 0,\; n\rightarrow \infty,\; \forall y\in Y. \label{qny}
\end{align}

\begin{theorem}\label{KarmaThm3} (Theorem 3 of \cite{karma1996approximation2}) If $F_n(k) = q'_n F(k)r_n$ for all $k\in \Lambda$ with $q'_n$ and $r_n$ satisfying \eqref{rnC}-\eqref{qny}, then, for sufficiently large $n$ and each $\lambda_n \in \sigma(F_n)\cap \Lambda_0$, it holds that
\begin{equation*}
|\lambda_0 - \lambda_n|\leqslant C(d_n d_n^*)^{1/\kappa},
\end{equation*}
where
\begin{align*}
d_n  = \max_{\substack{x\in G(F,\lambda_0)\\ \Vert x\Vert = 1}} \textnormal{dist}(x,r_n X_n),\quad
d_n^* = \max_{\substack{g\in G(F^*,\overline{\lambda_0})\\ \Vert g\Vert = 1}}\textnormal{dist}(g,(q'_n)^*Y^*_n).
\end{align*}
\end{theorem}

\subsection{Approximation of the Spectrum of $B^N_n(\cdot)$}
We apply the abstract spectral approximation theory to show the convergence of eigenvalues of $B^N_n(\cdot)$ to those of $B(\cdot)$. Assume that $\Lambda_0$ is a compact subset of $\Lambda$ so that $\partial\Lambda_0\subset\rho(B)$ and $\Lambda_0\cap \sigma(B)=\{\lambda_0\}$.

By the recurrence relation ${H^{(1)}_n}^{\prime}(z) = H^{(1)}_{n-1}(z) - nz^{-1}H^{(1)}_n(z)$ and the asymptotic formula of Hankel functions \cite{ernst1996finite}, we have that
\begin{equation}
\frac{k{H^{(1)}_n}^{\prime}(kR)}{H^{(1)}_n(kR)} = \frac{kH^{(1)}_{n-1}(kR)}{H^{(1)}_n(kR)} - \frac{n}{R},
\label{h}
\end{equation}
and
\begin{equation}
\left|\frac{kH^{(1)}_{n-1}(kR)}{H^{(1)}_n(kR)}\right|\leqslant C,\quad \left|\frac{k{H^{(1)}_n}^{\prime}(kR)}{H^{(1)}_n(kR)}\right|\leqslant C(1+n^2)^{1/2},\quad \forall n.
\label{hi}
\end{equation}
Similar to Theorems 3.1 and 3.3 of \cite{hsiao2011error}
for $k\in \Lambda$, we have the following estimate on $T^N(k)$.
\begin{lemma}\label{TTN}
$T(k)$ and $T^N(k)$, $k\in \Lambda$, are bounded from $H^{1/2}(\Gamma_R)$ to $H^{-1/2}(\Gamma_R)$. Furthermore,
\begin{equation*}
\Vert T(k)u - T^N(k)u\Vert_{H^{-1/2}(\Gamma_R)}\leqslant CN^{-s}\gamma_{1/2+s}(N,u)\Vert u\Vert_{H^{1+s}(\Omega_R)},
\end{equation*}
where $\gamma_{1/2+s}(N,u)\in [0,1]$, $\gamma_{1/2+s}(N,u)\rightarrow 0$ when $N\rightarrow \infty$, and $C$ does not depend on $N$ or $u$.
\end{lemma}
\begin{proof}
Recall that 
the norm on $H^s(\Gamma_R)$ is defined by 
\begin{equation*}
\Vert u\Vert_s^2 = \sum_{n=0}^{+\infty}{\!}^{{}^{\scriptstyle{\prime}}} (1+n^2)^s(|a_n(u)|^2+|b_n(u)|^2).
\end{equation*}
Hence by the estimate \eqref{hi}, 
\begin{eqnarray}
\Vert T(k)u\Vert_{H^{-1/2}(\Gamma_R)}^2 &=& \sum_{n=0}^{+\infty}{\!}^{{}^{\scriptstyle{\prime}}} (1+n^2)^{-1/2}\left|\frac{k{H^{(1)}_n}^{\prime}(kR)}{H^{(1)}_n(kR)}\right|^2(|a_n(u)|^2+|b_n(u)|^2) \nonumber \\
&\leqslant& C\Vert u\Vert_{H^{1/2}(\Gamma_R)}^2.
\label{C1}
\end{eqnarray}
Moreover 
\[\Vert T^N(k)u\Vert_{H^{-1/2}(\Gamma_R)}\leqslant \Vert T(k)u\Vert_{H^{-1/2}(\Gamma_R)}\leqslant C\Vert u\Vert_{H^{1/2}(\Gamma_R)}.
\]
Let $U^N$ be the linear 
span of $\cos(n\phi)$ and $\sin(n\phi)$, $n = 0,1,\dots,N$.
Define $\pi^N : L^2(\Gamma_R)\rightarrow U^N$ as the $L^2(\Gamma_R)$-projection.
It holds that $T^N(k)u = T(k)\pi^N u$. For $r\geqslant t$
\begin{eqnarray*}
\Vert u - \pi^Nu\Vert_{H^{t}(\Gamma_R)}^2 &=& \sum_{n=N+1}^{+\infty}(1+n^2)^t(|a_n(u)|^2+|b_n(u)|^2) \\
&\leqslant& N^{-2r+2t}\Vert u - \pi^N u\Vert_{H^{r}(\Gamma_R)}^2.
\end{eqnarray*}
Denoting $\gamma_{r}(N,u) = \Vert u - \pi^N u\Vert_{H^{r}(\Gamma_R)}/\Vert u\Vert_{H^{r}(\Gamma_R)}$, we have that
\begin{equation}
\Vert u - \pi^N u\Vert_{H^t(\Gamma_R)} \leqslant N^{-r+t}\gamma_r(N,u)\Vert u\Vert_{H^r(\Gamma_R)}.
\label{pi}
\end{equation}
Therefore, by the boundedness of $T(k)$ and the trace inequality, it holds
for $r = 1/2+s$, $s>0$ that
\begin{eqnarray}
&&\Vert T(k)u - T^N(k)u\Vert_{H^{-1/2}(\Gamma_R)} \leqslant C\Vert u-\pi^Nu\Vert_{H^{1/2}(\Gamma_R)} \nonumber
\\ &\leqslant& CN^{-s}\gamma_{1/2+s}(N,u)\Vert u\Vert_{H^{1/2+s}(\Gamma_R)} \leqslant CN^{-s}\gamma_{1/2+s}(N,u)\Vert u\Vert_{H^{1+s}(\Omega_R)}.
\label{C2}
\end{eqnarray}
This completes the proof.
\end{proof}


Next we consider the approximation of $B^N(\cdot)$ to $B(\cdot)$.
Take $X = Y = X_N = Y_N = V$, and let the operators $p_N : V\rightarrow V$
and $q_N : V\rightarrow V$ be identities.

By the equality \eqref{h},
we write $B(k) = A - K(k)$, where $A:V\rightarrow V$ and $K(k):V\rightarrow V$ are
defined respectively by
\[
(Au,v)_1 = (\nabla u,\nabla v) + (u,v) + 
\pi\sum_{n=0}^{+\infty}{\!}^{{}^{\scriptstyle{\prime}}} n(a_n(u)\overline{a_n(v)} + b_n(u)\overline{b_n(v)})
\]
for all $u, v\in V$ and
\[
(K(k)u,v)_1 = ((k^2n(x)+1)u,v) + \pi R\sum_{n=0}^{+\infty}{\!}^{{}^{\scriptstyle{\prime}}}
\frac{kH^{(1)}_{n-1}(kR)}{H^{(1)}_n(kR)}(a_n(u)\overline{a_n(v)} + b_n(u)\overline{b_n(v)}).
\]
for all $u, v\in V$.
Define $A^N: V\rightarrow V$ and $K^N(k): V\rightarrow V$ respectively by truncating the series to the $N$-th term such that
\[
(A^Nu,v)_1 = (\nabla u,\nabla v) + (u,v) + 
\pi\sum_{n=0}^{N}{\!}^{{}^{\scriptstyle{\prime}}} n(a_n(u)\overline{a_n(v)} + b_n(u)\overline{b_n(v)})
\]
for all $u, v\in V$ and
\[
(K^N(k)u,v)_1 = ((k^2n(x)+1)u,v) + \pi R\sum_{n=0}^{N}{\!}^{{}^{\scriptstyle{\prime}}}
\frac{kH^{(1)}_{n-1}(kR)}{H^{(1)}_n(kR)}(a_n(u)\overline{a_n(v)} + b_n(u)\overline{b_n(v)})
\]
for all $u, v\in V$.
Consequently,  $B^N(k) = A^N - K^N(k)$.
\begin{lemma}\label{LemmaBhF}
The operator function $B(\cdot):\Lambda\rightarrow \mathcal{L}(V)$ is holomorphic 
Fredholm with index zero. The resolvent set of $B(\cdot)$ is non-empty.
\end{lemma}
\begin{proof}
Let $k\in \Lambda$.
Taking the derivative of \eqref{h} and making use of the recurrence formula and the estimate
\eqref{hi}, we have that
$T'(k) : H^{1/2}(\Gamma_R)\rightarrow H^{-1/2}(\Gamma_R)$ is bounded.
Then taking the derivative of \eqref{B}, we see that $B(\cdot) : \Lambda\rightarrow \mathcal{L}(V)$
is holomorphic with $B'(\cdot)$ satisfying
\begin{equation*}
(B'(k)u,v)_1 = -(2kn(x)u,v) - \langle T'(k)u,v\rangle \quad \forall v\in V.
\end{equation*}
Inequality \eqref{hi} implies that $A$ is bounded and the extension $K(k) : H^{1/2+\epsilon}(\Omega_R)\rightarrow V$ is also bounded. Thus $K(k)$ itself is compact.
Moreover, since $(Au,u)_1\geqslant \Vert u\Vert_1^2$ for all $u\in V$, $A$ is invertible.
Therefore $B(k)$ is Fredholm of index zero.
The uniqueness of the scattering problem for $\text{Im}(k)\ge 0$ implies that
$\rho(B)$ is non-empty.
\end{proof}

\begin{lemma}\label{BNhF}
The operator functions 
$B^N(\cdot) : \Lambda\rightarrow \mathcal{L}(V)$ are holomorphic Fredholm with index zero,
and for each compact $\Lambda_1\subset \Lambda$,
$\Vert B^N(k)\Vert\leqslant C$ for all $N$ and for all $k\in \Lambda_1$.
\end{lemma}
\begin{proof}
Similar to the proof of Lemma~\ref{LemmaBhF}, it can be shown that $B^N(\cdot):\Lambda\rightarrow V$ is holomorphic. By Lemma~\ref{TTN},
\begin{eqnarray*}
\Vert B^N(k)u\Vert_1 &\leqslant& C\Vert u\Vert_1 + \Vert T^N(k)u\Vert_{H^{-1/2}(\Gamma_R)} \\
&\leqslant& C\Vert u\Vert_1 + C\Vert u\Vert_{H^{1/2}(\Gamma_R)} \\
&\leqslant& C\Vert u\Vert_1,
\end{eqnarray*}
where $C$ is independent of $N$. Furthermore, if $k\in \Lambda_1$,
then $C$ can be chosen to be independent of $k$.
Again, similar to Lemma~\ref{LemmaBhF}, we can show that $A^N$ is bounded and $K^N(k)$
is compact. Moreover, $(A^Nu,u)_1\geqslant \Vert u\Vert_1^2$ for all $u\in V$.
Thus $A^N$ is invertible. Therefore, $B^N(k)$ is Fredholm of index zero for $k\in \Lambda$.
\end{proof}

\begin{theorem}\label{BN}
Assume that $G(B,\lambda_0)\subset H^{1+s}(\Omega_R)$ for some
$s>0$. For sufficiently large $N$ and for each
$\lambda^N\in \sigma(B^N)\cap \Lambda_0$, 
$\nu(\Lambda_0,B) = \nu(\Lambda_0,B^N)$, and
\begin{equation*}
|\lambda_0 - \lambda^N|\leqslant CN^{-s/\kappa},
\end{equation*}
where $\kappa$ is the order of the pole $\lambda_0$ of $B^{-1}(\cdot)$.
\end{theorem}
\begin{proof}
Let $k\in \Lambda$.
By the definition of $A$ and $A^N$, and \eqref{pi}, we have that
\begin{align*}
|(Au,v)_1 - (A^Nu,v)_1| & = \left|\pi\sum_{n=N+1}^{+\infty}n(a_n(u)\overline{a_n(v)} + b_n(u)\overline{b_n(v)})\right|
\\
& \leqslant C\Vert u-\pi^N u\Vert_{H^{1/2}(\Gamma_R)}
\Vert v \Vert_{H^{1/2}(\Gamma_R)} \\ 
&\leqslant C\gamma_{1/2}(N,u)\Vert u\Vert_{H^{1/2}(\Gamma_R)}\Vert v\Vert_{H^{1/2}(\Gamma_R)} \\ & \leqslant C\gamma_{1/2}(N,u)\Vert u\Vert_1\Vert v\Vert_1.
\end{align*}
Hence $\Vert Au - A^N u\Vert_1\leqslant \gamma_{1/2}(N,u)\Vert u\Vert_1\rightarrow 0$ as $N\rightarrow\infty$ for all $u\in V$. Moreover,  $(A^Nu,u)_1\geqslant \Vert u\Vert_1^2$ for all $u\in V$, thus $A^N$ are invertible with $\Vert (A^N)^{-1}\Vert \leqslant 1$ for all $N$. Therefore, $A^N$ converges stably to $A$. On the other hand, by the definition of $K(k)$ and $K^N(k)$, and the inequalities \eqref{hi} and \eqref{pi}, we have that
\begin{align*}
|(K(k)u,v)_1 - (K^N(k)u,v)_1| & = \left|\pi R\sum_{n=N+1}^{+\infty}\frac{kH^{(1)}_{n-1}(kR)}{H^{(1)}_n(kR)}(a_n(u)\overline{a_n(v)} + b_n(u)\overline{b_n(v)})\right| 
\\
& \leqslant
C\Vert u - \pi^N u\Vert_{H^{-1/2}(\Gamma_R)}\Vert v\Vert_{H^{1/2}(\Gamma_R)} \\
& \leqslant CN^{-1}\Vert u\Vert_{H^{1/2}(\Gamma_R)}\Vert v\Vert_{H^{1/2}(\Gamma_R)} \\
&\leqslant CN^{-1}\Vert u\Vert_1\Vert v\Vert_1.
\end{align*}
Hence $\Vert K(k)u - K^N(k)u\Vert_1\leqslant CN^{-1}\Vert u\Vert_1$
for all $u\in V$. Thus $K^N(k)$ converges uniformly to $K(k)$.
Due to Corollary~\ref{EH}, for sufficiently large $N$ and $\lambda^N\in \sigma(B^N)\cap \Lambda_0$, it holds that $\nu(\Lambda_0,B) = \nu(\Lambda_0,B^N)$ and
\begin{equation*}
|\lambda_0 - \lambda^N|\leqslant C\sup_{\substack {u\in G(B,\lambda_0)\\
\Vert u\Vert_1 = 1,\; k\in \partial\Lambda_0}}\Vert B(k)u - B^N(k)u\Vert_1^{1/\kappa}.
\end{equation*}
By definition,
\begin{align*}
|(B(k)u,v)_1-(B^N(k)u,v)_1| & = |\langle T(k)u - T^N(k)u,v\rangle|
\\ &\leqslant \Vert T(k)u - T^N(k)u\Vert_{H^{-1/2}(\Gamma_R)}\Vert v\Vert_{H^{1/2}(\Gamma_R)}.
\end{align*}
Therefore, by the trace inequality, Lemma 1 and 
that $G(B,\lambda_0)$ is finite dimensional, we have
\begin{equation*}
|\lambda_0 - \lambda^N|\leqslant C\sup_{\substack {u\in G(B,\lambda_0)\\
\Vert u\Vert_1 = 1,\; \lambda\in \partial\Lambda_0}}\Vert T(k)u - T^N(k)u\Vert_{H^{-1/2}(\Gamma_R)}^{1/\kappa}\leqslant CN^{-s/\kappa}.
\end{equation*}
The proof is complete.
\end{proof}

Now we consider the approximation of $B^N_n(\cdot)$ to $B^N(\cdot)$ as $n\rightarrow \infty$.
Let $X = Y = V$, $X_n = Y_n = V_n \subset V$, and take $p_n: V\rightarrow V_n$
as the $V$-projection, and $q_n = p_n$.
We split $B^N_n(k)$ in the same way as $B^N(k)$.
By the definition of $B^N_n(k)$, it holds that $(B^N_n(k)u_n,v_n)_1 = (B^N(k)u_n,v_n)_1$
for all $u_n, v_n\in V_n$.
Define $A^N_n:V_n\rightarrow V_n$ and $K^N_n(k):V_n\rightarrow V_n$ by
\[
(A^N_n u_n,v_n)_1 = (A^N u_n,v_n)_1 \quad \forall u_n, v_n\in V_n
\]
and
\[
(K^N_n(k) u_n,v_n)_1 = (K^N(k) u_n,v_n)_1 \quad
\forall u_n, v_n\in V_n,
\]
respectively. Then $B^N_n(k) = A^N_n - K^N_n(k)$.
\begin{theorem}\label{BNn}
Assume that $G(B^N,\lambda^N)\subset H^{1+s_N}(\Omega_R)$ for some $0<s_N\leqslant 1$. For sufficiently large $N$ and $n$ and for each
$\lambda^N_n\in \sigma(B^N_n)\cap \Lambda_0$, 
$\nu(\Lambda_0,B) = \nu(\Lambda_0,B^N_n)$, and
\begin{equation*}
|\lambda^N - \lambda^N_n|\leqslant Ch_n^{2s_N/\kappa_N},
\end{equation*}
where $\kappa_N$ is the order of the pole $\lambda^N$ of $(B^N)^{-1}(\cdot)$.
\end{theorem}
\begin{proof}
By Theorem \ref{BN}, the resolvent set of $B^N(\cdot)$ is non-empty.
Since $V_n$ is finite dimensional, $B^N_n : \Lambda\rightarrow \mathcal{L}(V_n)$ 
are holomorphic Fredholm with index zero.
Analogous to Lemma \ref{BNhF}, for each compact $\Lambda_1\subset \Lambda$, 
$\Vert B^N_n(k)\Vert\leqslant C$ holds for all $n$ and $N$ and all 
$k\in \Lambda_1$. 
Moreover, the boundedness of $\Vert (A^N_n)^{-1}\Vert \leqslant C$ 
for all $N$ and $n$ holds for the same reason as that of $(A^N)^{-1}$.
Let $k\in \Lambda$.
By definition,
\begin{equation*}
B^N_n(k) = q_n B^N(k) r_n, \quad
A^N_n = q_n A^N r_n,\quad
K^N_n(k) = q_n K^N(k) r_n,
\end{equation*}
with $r_n : V_n \rightarrow V$ the embedding.
Hence we have that
\begin{align*}
\Vert q_n A^N u - A^N_n p_n u\Vert_1 &= \Vert q_n A^N(u - p_n u)\Vert_1 \\
& \leqslant \Vert q_n\Vert \Vert A^N\Vert \Vert u - p_n u\Vert_1 \rightarrow 0,\; n\rightarrow \infty,\; \forall u\in V,
\end{align*}
and, by the boundedness of $K^N(k):H^{1/2+\epsilon}(\Omega_R)\rightarrow V$,
\begin{align*}
\Vert q_n K^N(k) u - K^N_n p_n u\Vert_1 & = \Vert q_n K^N(k)(u - p_n u)\Vert_1
\\
& \leqslant \Vert q_n\Vert \Vert K^N(k)\Vert_{H^{1/2+\epsilon}(\Omega_R)\rightarrow V}
\Vert u - p_n u\Vert_{1/2+\epsilon}\\ &\leqslant Ch_n^{1/2-\epsilon}\Vert u\Vert_1,\; \forall u\in V.
\end{align*}
Therefore, $A^N_n$ converges stably to $A^N$ and $K^N_n(k)$ converges
uniformly to $K^N(k)$. By Corollary~\ref{TTN}, $\nu(\Lambda_0,B^N)
= \nu(\Lambda_0,B^N_n)$ 
for sufficiently large $n$.

Finally, we apply Theorem~\ref{KarmaThm3}.
Let $q'_n = q_n$ and $r_n$ be defined as above such that \eqref{rnC}-\eqref{qny} are satisfied.
In addition,
\begin{equation*}
(q_n u,v_n)_1= (u,r_nv_n)_1, \quad \forall u\in V,\; \forall v_n\in V_n.
\end{equation*}
Hence $q_n^* = r_n$.
For each $u, v\in V$ it holds
\begin{align*}
\langle T^N(k)u,v\rangle &= \langle  \sum_{n=0}^{N}{\!}^{{}^{\scriptstyle{\prime}}}\frac{k{H^{(1)}_n}^{\prime}(kR)}{H^{(1)}_n(kR)}(a_n(u)\cos(n\theta) + b_n(u)\sin(n\theta)),v\rangle
\\
&= \pi R \sum_{n=0}^{N}{\!}^{{}^{\scriptstyle{\prime}}} \frac{k{H^{(1)}_n}^{\prime}(kR)}{H^{(1)}_n(kR)} (a_n(u)a_n(\overline{v})
+ b_n(u)b_n(\overline{v})) \\
&= \langle u,\overline{T^N(k)\overline{v}}\rangle.
\end{align*}
Therefore, $(T^N)^*(k)f = T^N(\overline{k})^*f = \overline{T^N(\overline{k})\overline{f}}$.
So $(B^N)^*(k)u = \overline{B^N(\overline{k})\overline{u}}$.
Hence $((B^N)^*)^{(\ell)}(k)u = \overline{(B^N)^{(\ell)}(\overline{k})\overline{u}}$.
Consequently,
$G((B^N)^*,\overline{\lambda^N}) = \overline{G(B^N,\lambda^N)}$.
Altogether, we have that
\begin{equation*}
d^*_n = d_n = \sup_{\substack{u\in G(B^N,\lambda^N)\\ \Vert u\Vert_1=1}}
\inf_{u_n\in V_n}\Vert u - u_n\Vert_1\leqslant Ch_n^{s_N}.
\end{equation*}
By Theorem~\ref{KarmaThm3}, the inequality $|\lambda^N - \lambda^N_n|\leqslant Ch_n^{2s_N/\kappa_N}$ holds.
\end{proof}

\section{Numerical Experiments}
We present numerical examples to demonstrate the performance of the proposed method, including its convergence order, the impact of the truncation order in the DtN mapping, the effect of the radius of the truncated domain, and the validation of the theoretical results in literature \cite{Popov1999, galkowski2018quantum, moiola2019acoustic, dyatlov2019mathematical}.

The construction of the conforming finite element spaces $V_h \subset V$ is as follows. First a regular triangulation $\tilde{\mathcal{T}}_{h}$ for $\Omega_R$ is generated with mesh size $h$. Note that the union of all triangles is a polygon. For each triangle, its intersection with a curved $\Gamma$ or $\Gamma_R$ is either empty, or contains only a vertex or two vertices of the triangle. Based on $\tilde{\mathcal{T}}_h$, a partition $\mathcal{T}_h$ of $\Omega_R$ is obtained by allowing its elements 
to have one curved edge, that is by replacing a straight edge having both
vertices on $\Gamma$ or $\Gamma_R$ by the corresponding curve segment.
For each curvilinear element $K\in \mathcal{T}_h$, we denote by
$\tilde{F}_{\tilde{K}}$ the mapping from the corresponding straight
element $\tilde{K}\in \tilde{\mathcal{T}}_h$ to $K$.
Then the conforming finite element space is defined by
\begin{equation}\label{Vh}
V_h = \{u_h\in H^1(\Omega_R)\,:\, u_h|_K\circ \tilde{F}_{\tilde{K}} = p_h, \, \text{for some\,} \, p_h\in P_1(\tilde{K})\}.
\end{equation}
For the definition of the mapping $\tilde{F}_{\tilde{K}}$
and the approximation property of $V_h$, see Section 8.3 of \cite{Monk03}.

    Note that, for second order PDEs and linear scalar elements, simple curved domains can be approximately covered by triangular meshes. The “variational crimes” of not using the exact computational domain or exact quadratures are admissible and do not spoil convergence. For higher-order elements, this crude boundary approximation impacts the convergence rate. Discussions along this line are classical and can be found in \cite{Ciarlet1972, Fix1972, zlamal1973, Dubios1990SINUM}, Chp. III of \cite{Braess2007}, Chp. 10 of \cite{BrennerScott2008} and Section 8.3 of \cite{Monk03}.

\subsection{Example 1}
Let $\Omega$ be the unit disk. We consider the transmission problem with the incident field given by
\[
	u^i(r,\theta)=J_\ell(k r){\rm e}^{{i }\ell \theta},\quad \theta\in[0,2\pi),
\]
 where $r=|x|$, $J_\ell$ is the first kind Bessel function of order $\ell$  ($\ell=0$, $1$, $2$, $\cdots$).
 The scattered field can be written as
\[
 	u^s(r,\theta)=\dfrac{C_\ell(k;k \sqrt{n_i})}{W_\ell(k;k \sqrt{n_i})}H_\ell^{(1)}(k r){\rm e}^{i \ell \theta},\quad r > 1, \quad \theta\in[0,2\pi),
\]
 where $H_\ell^{(1)}$ is the Hankel function of the first kind of order $\ell$, 
 \[
 	W_\ell(k; k \sqrt{n_i})=
 	\left|\begin{array}{cc}
 		J_\ell(k \sqrt{n_i})&-H_\ell^{(1)}(k)\\
 		k \sqrt{n_i} J_\ell'(k \sqrt{n_i})&- k H_\ell^{(1)}{'} (k)
 	\end{array}\right|,
 \]
 and
  \[
 	C_\ell(k;k \sqrt{n_i})=
 	\left|\begin{array}{cc}
 	J_\ell(k \sqrt{n_i})&J_\ell(k)\\
 		k \sqrt{n_i} J_\ell'(k \sqrt{n_i})&k J_\ell'(k)
 	\end{array}\right|.
 \]
 The exact scattering poles are $k$'s such that $W_\ell(k; k \sqrt{n_i})=0$
for $\ell = 0,1,\dots$

A series of meshes are generated with the size of the coarsest mesh $h_1 \approx 0.04\pi$. Subsequent meshes have mesh sizes $h_j = h_{j-1}/2$. We take $R = 1.25$ for $\Omega_R$ and $N = 20$ for the truncation order of the DtN mapping. Let $n_i = 4$. We compute the scattering poles in the region 
\[
\Lambda_0 := \{z\in \mathbb{C}\,;\, 0<\text{Re}z<4,
-4<\text{Im}z<0\}.
\] 
The computed poles are ordered by their modulus. We list some computed poles using the mesh size $h = h_5$:
\begin{align*}
&0.4367 - 0.3039i, \quad 1.1155 - 0.2396i, \quad  1.1155 - 0.2396i, \quad 1.7563 - 0.1744i, 
\\
& 1.7563 - 0.1744i, \quad  0.4509 - 1.7933i, \quad  0.4509 - 1.7933i, \quad 1.9778 - 0.2791i,  
\\
&2.3841 - 0.1217i, \quad  2.3841 - 0.1217i, \quad  1.3476 - 2.2112i, \quad 1.3476 - 2.2112i.
\end{align*}
Two poles are isolated, while the others form closely spaced pairs, suggesting multiplicities of $1$ and $2$, respectively.
For each pair of poles, we take the average of them when we compute the errors. Denote by $k$ an exact pole, by $k_j$ its approximation using mesh size $h_j$. We define the error and convergence order as
\[
e_j := k_j - k, \quad 
{\rm CO}:=\log_2(|e_{j-1}|/|e_j|).
\]
The results are shown in Table~\ref{T:ex1_order} for three poles. The first row shows the values computed by the finest mesh ($h_5$). We observe the second order convergence as $P_1$ is used to define $V_h$ in \eqref{Vh}.

\begin{table}[!htp]
\centering
\begin{tabular}{r|c|c|c|c|c|c}
\hline
 $k_5$& $0.4367 -0.3039i$ & & $1.1155 -0.2396i$ & & $1.7563 -0.1744i$ &
 \\
\hline
$e_2$ & $-$2.4e-05 $-$ 4.4e-05$i$ &  & 3.8e-04 $-$ 5.3e-04$i$ &  & 1.9e-03 $-$ 1.1e-03$i$ & 
\\
\hline
$e_3$ & $-$6.0e-06 $-$ 1.1e-05$i$ & 1.99 & 9.6e-05 $-$ 1.3e-04$i$ & 1.98 & 4.8e-04 $-$ 2.9e-04$i$ & 1.98 
\\
\hline
$e_4$ & $-$1.5e-06 $-$ 2.8e-06$i$ & 1.98 & 2.4e-05 $-$ 3.3e-05$i$ & 1.99 & 1.2e-04 $-$ 7.3e-05$i$ & 1.99
\\
\hline
$e_5$ & $-$3.8e-07 $-$ 7.1e-07$i$ & 1.99 & 6.0e-06 $-$ 8.4e-06$i$ & 2.00 & 3.0e-05 $-$ 1.8e-05$i$ & 2.00
\\
\hline
\end{tabular}
\caption{The convergence order of the computed poles for the unit disk with inhomogeneity $n_i = 4$.}
\label{T:ex1_order}
\end{table}
 

Next we check the dependence of the relative error on the radius $R$ for $\Omega_R$. Let $k_5$ be a computed pole on the mesh with mesh size $h_5$. The relative error is defined as $RE_5 = |k_5 - k|/|k|$. We compute the relative errors for  $R = 1.25, 1.2, 1.15, 1.1, 1.05$. Listed in Table \ref{T:ex1_R} are the results for three poles. The choice of $R$ has little impact on the error; thus, smaller $R$ should be used to reduce computational cost.

\begin{table}[!htp]
\centering
\begin{tabular}{c|c|c|c}
\hline
$k_5$ & $0.4509 - 1.7933i$ & $1.9778 - 0.2791i$ & $2.3841 - 0.1217i$
\\
\hline
$R = 1.25$ & 2.5424e-05 & 3.2653e-05 & 3.7006e-05
\\
\hline
$R = 1.20$ & 2.2160e-05 & 3.2792e-05 & 3.6109e-05
\\
\hline
$R = 1.15$ & 1.9395e-05 & 3.2904e-05 & 3.4951e-05
\\
\hline
$R = 1.10$ & 1.7005e-05 & 3.2985e-05 & 3.3413e-05
\\
\hline
$R = 1.05$ & 1.4957e-05 & 3.3043e-05 & 3.1331e-05
\\
\hline
\end{tabular}
\caption{The relative error  for different $R$'s for the unit disk with $n_i = 4$ and mesh size $h = h_5$.}
\label{T:ex1_R}
\end{table}

Now we check the impact of truncation order $N$ for the DtN mapping on the accuracy. In Fig.~\ref{F:ex1_N}, we plot the relative errors of the computed poles for $N$ varying from $5$ to $20$, with each line connecting the errors of the computed poles that correspond to the same exact pole. The results show that, once $N$ is sufficiently large ($N\ge 5$ in this case), further increases in $N$ have a negligible effect on the computed poles.
\begin{figure}[!htp]
\centering
\includegraphics[width=0.6\linewidth]
{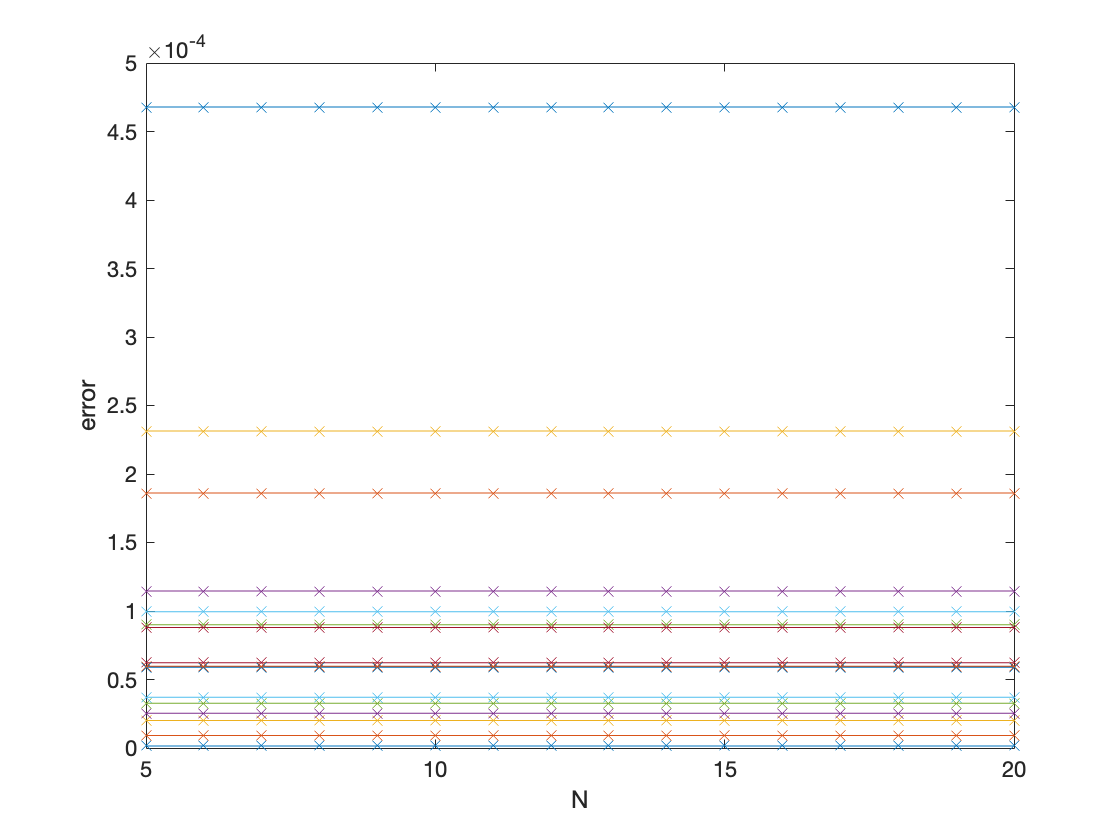}
\caption{The error of the computed poles for $N = 5,\dots,20$ for the unit disk with $n_i = 4$ ($h = h_5$).}
\label{F:ex1_N}
\end{figure}

We carry out the same experiment for a different index of refraction $n_i = 0.25$. The  computed poles on the finest mesh are:
\begin{align*}
&0.4597 - 1.3756i, \quad 0.4597 - 1.3756i, \quad 1.3697 - 1.7654i, \quad 1.3697 - 1.7654i, \\ & 0.4475 - 2.7233i, \quad 0.4475 - 2.7233i, \quad 2.2830 - 2.0533i, \quad 2.2830 - 2.0533i, \\ & 1.3400 - 3.2270i, \quad 1.3400 - 3.2270i, \quad 3.2028 - 2.2873i, \quad 3.2028 - 2.2873i.
\end{align*}
The error and the convergence order are shown in Table~\ref{T:ex1b_order} for three poles. Second order convergence is observed.
\begin{table}[!htp]
\centering
\begin{tabular}{r|c|c|c|c|c|c}
\hline
 $k_5$& $0.4597 - 1.3756i$ & & $1.3697 - 1.7654i$ & & $0.4475 - 2.7233i$ &
 \\
 \hline
$e_2$ & 2.2e-04 $-$ 1.9e-04$i$ &  & 9.1e-04 $-$ 2.9e-03$i$ &  & 2.6e-03 $-$ 3.1e-02$i$ &  
\\
\hline
$e_3$ & 6.1e-05 $-$ 4.5e-05$i$ & 1.94 & 3.0e-04 $-$ 6.8e-04$i$ & 2.02 & 7.4e-04 $-$ 7.8e-03$i$ & 2.00 
\\
\hline
$e_4$ & 1.5e-05 $-$ 1.0e-05$i$ & 2.02 & 7.7e-05 $-$ 1.7e-04$i$ & 2.02 & 1.9e-04 $-$ 1.9e-03$i$ & 2.01 
\\
\hline
$e_5$ & 3.9e-06 $-$ 2.5e-06$i$ & 2.02 & 1.9e-05 $-$ 4.1e-05$i$ & 2.01 & 4.7e-05 $-$ 4.8e-04$i$ & 2.01 
\\
\hline
\end{tabular}
\caption{The convergence order of the computed poles for the unit disk with $n_i = 0.25$.}
\label{T:ex1b_order}
\end{table}

The relative error for different $R$'s and $N$'s
are shown in Table \ref{T:ex1b_R} and Fig.~\ref{F:ex1b_N}, respectively. Again, the results indicate that the choices of $R$ and $N$ have little impact on the accuracy for $R$ ranging from
1.05 to 1.25 and $N$ from 6 to 20.
\begin{table}[!htp]
\centering
\begin{tabular}{c|c|c|c}
\hline
$k_5$ & $2.2830 - 2.0533i$ & $1.3400 - 3.2270i$ & $3.2028 - 2.2873i$
\\
\hline
$R = 1.25$ & 5.0727e-05 & 4.1005e-04 & 9.5687e-05
\\
\hline
$R = 1.20$ & 4.3417e-05 & 2.8025e-04 & 8.0809e-05
\\
\hline
$R = 1.15$ & 3.7752e-05 & 1.9037e-04 & 6.8752e-05
\\
\hline
$R = 1.10$ & 3.3164e-05 & 1.2817e-04 & 5.8420e-05
\\
\hline
$R = 1.05$ & 2.9361e-05 & 8.6144e-05 & 4.9411e-05
\\
\hline
\end{tabular}
\caption{The relative errors for different $R$'s for the unit disk with
$n_i = 0.25$ and mesh size $h = h_5$.}
\label{T:ex1b_R}
\end{table}

\begin{figure}[!htp]
\centering
\includegraphics[width=0.6\linewidth]
{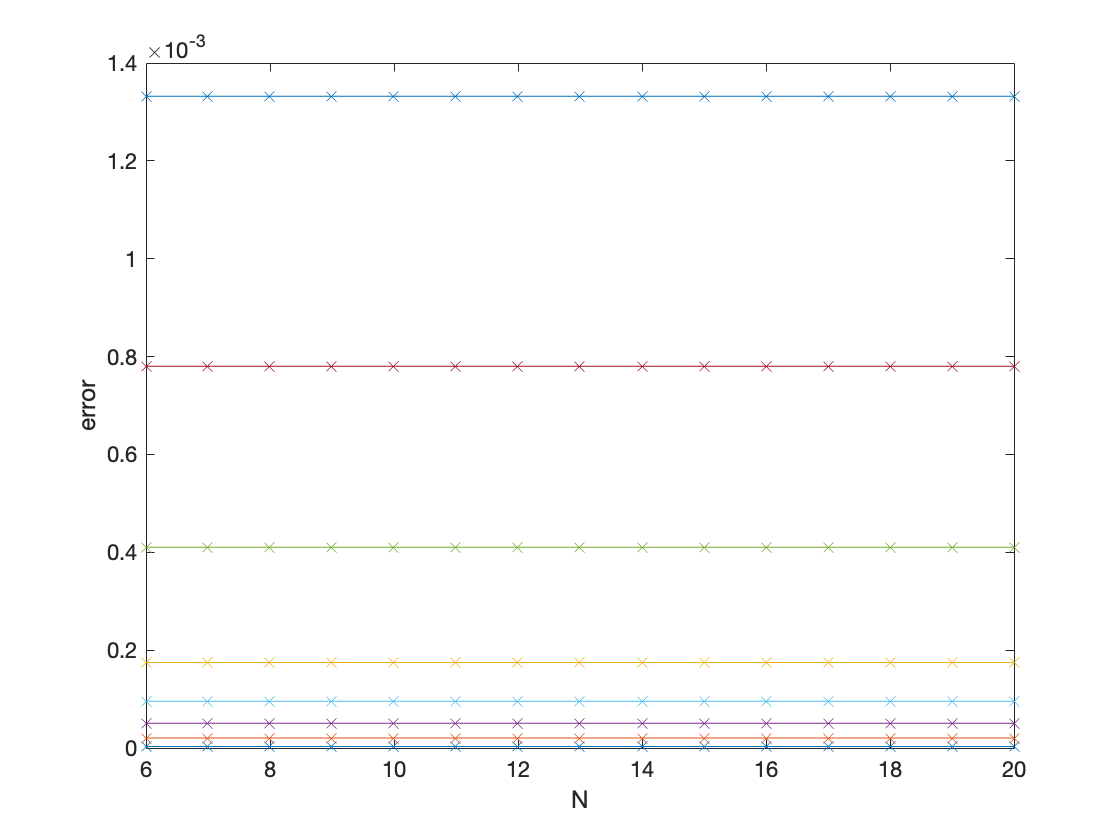}
\caption{The error for $N = 6,\dots,20$ for the unit disk with $n_i = 0.25$ ($h = h_5$).}
\label{F:ex1b_N}
\end{figure}


\subsection{Example 2}
In this example, we compute the scattering poles for the unit square $(-0.5,0.5)^2$. The setting are the same as the previous example except that $R = 0.8$ for $\Omega_R$. We list below the first several poles computed with $n_i=4$ and on the mesh $h = h_5$.
\begin{align*}
&0.7719 - 0.5435i, \quad 1.9724 - 0.4386i, \quad 1.9724 - 0.4386i, \quad 2.9121 - 0.3073i, \\ &0.7201 - 2.9037i, \quad 3.3601 - 0.3411i, \quad 3.4397 - 0.5360i, \quad 0.8476 - 3.3959i, \\& 2.3674 - 3.8539i, \quad 2.3674 - 3.8539i.
\end{align*}
The convergence orders for three poles computed on meshes with $h = h_1, \dots,  h_5$ are shown in Table~\ref{T:ex2_order}. Since the exact poles are not available, the error is defined as $e_j = k_j - k_{j+1}$ on two meshes. The convergence order is defined as $\log_2(|e_{j-1}|/|e_j|)$. Second order convergence is achieved.
\begin{table}[!htp]
\centering
\begin{tabular}{c|c|c|c|c|c|c}
\hline
$k_5$ & $0.7719 - 0.5435i$ & & $1.9724 - 0.4386i$ & & $2.9121 - 0.3073i$
\\
\hline
$e_1$ & $-$4.8e-04 $-$ 7.2e-04$i$ & & 6.0e-03 $-$ 9.0e-03$i$ & & 2.5e-02 $-$ 1.7e-02$i$ &
\\
\hline
$e_2$ & $-$1.2e-04 $-$ 1.8e-04$i$ & 1.97 & 1.5e-03 $-$ 2.3e-03$i$ & 1.96 & 6.3e-03 $-$ 4.3e-03$i$ & 1.98
\\
\hline
$e_3$ & $-$3.3e-05 $-$ 4.9e-05$i$ & 1.92 & 4.0e-04 $-$ 5.9e-04$i$ & 1.96 & 1.6e-03 $-$ 1.1e-03$i$ & 1.96
\\
\hline
$e_4$ & $-$8.2e-06 $-$ 1.2e-05$i$ & 1.98 & 1.0e-04 $-$ 1.5e-04$i$ & 2.00 & 4.1e-04 $-$ 2.7e-04$i$ & 2.00
\\
\hline
\end{tabular}
\caption{The error and
the convergence order for the unit square with $n_i = 4$.}
\label{T:ex2_order}
\end{table}

We carry out the computation for $n_i = 0.25$ on meshes with $h = h_1,\dots, h_5$. The poles in $\Lambda_0$ are
\begin{align*}
&0.7968 - 2.4042i, \quad 0.7968 - 2.4042i, \quad 2.2210 - 2.8852i, \quad 2.5953 - 3.3590i, \\& 3.9984 - 3.5622i,\quad 3.9984 - 3.5622i.
\end{align*}
The results for three poles are shown in Table~\ref{T:ex2b_order}. Second order convergence is obtained.
\begin{table}[!htp]
\centering
\begin{tabular}{c|c|c|c|c|c|c}
\hline
$k_5$ & $0.7968 - 2.4042i$ & & $2.2210 - 2.8852i$ & & $2.5953 - 3.3590i$
\\
\hline
$e_1$ & 6.3e-03 $-$ 4.7e-03$i$ & & 3.1e-02 $-$ 4.2e-02$i$ & & 1.2e-01 $-$ 1.5e-02$i$ &
\\
\hline
$e_2$ & 2.1e-03 $-$ 1.4e-03$i$ & 1.65 & 6.9e-03 $-$ 1.0e-02$i$ & 2.07 & 4.8e-02 $-$ 3.7e-04$i$ & 1.36
\\
\hline
$e_3$ & 4.7e-04 $-$ 3.0e-04$i$ & 2.16 & 1.7e-03 $-$ 2.4e-03$i$ & 2.06 & 1.0e-02 $-$ 2.0e-03$i$ & 2.20
\\
\hline
$e_4$ & 1.2e-04 $-$ 7.3e-05$i$ & 1.99 & 4.2e-04 $-$ 5.9e-04$i$ & 2.03 & 2.7e-03 $-$ 4.8e-04$i$ & 1.97
\\
\hline
\end{tabular}
\caption{The error and the convergence order for the unit square with $n_i = 0.25$.}
\label{T:ex2b_order}
\end{table}

Next we show the results for poles with larger real parts. In Fig.~\ref{F:ex2_inh4} we plot the computed poles 
for $n_i = 4$ with $h = h_5$ and $N = 30$. 
Other parameters remain the same as in previous experiments. One observes a sequence of poles approach to real axis from below as their real parts increase, indicating that they approach the real axis. Note that it is proved in \cite{Popov1999} that, if $\Omega$ is smooth and strictly convex, for index of refraction greater than $1$, there exists an infinite sequence of resonances tending to the real axis.

\begin{figure}[!htp]
\centering
\includegraphics[width=0.6\linewidth]
{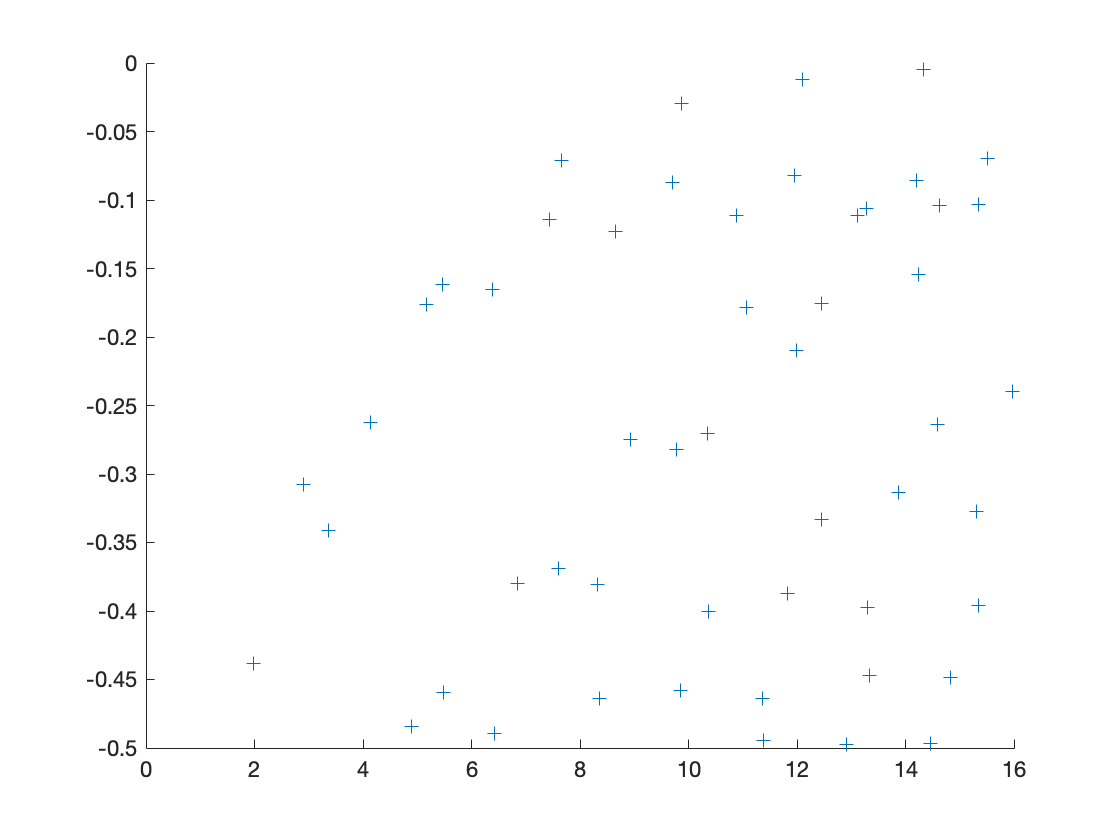}
\caption{The computed poles of the unit square with $n_i = 4$ with mesh size $h = h_5$.}
\label{F:ex2_inh4}
\end{figure}

We perform the same computation for $n_i = 0.25$. The computed poles are plotted in Fig.~\ref{F:ex2_inh0.25}. It is observed that there is a gap between the computed poles and the real axis. The phenomenon coincides with the theory on pole-free zone in the complex plane in \cite{moiola2019acoustic}.
\begin{figure}[!htp]
\centering
\includegraphics[width=0.6\linewidth]
{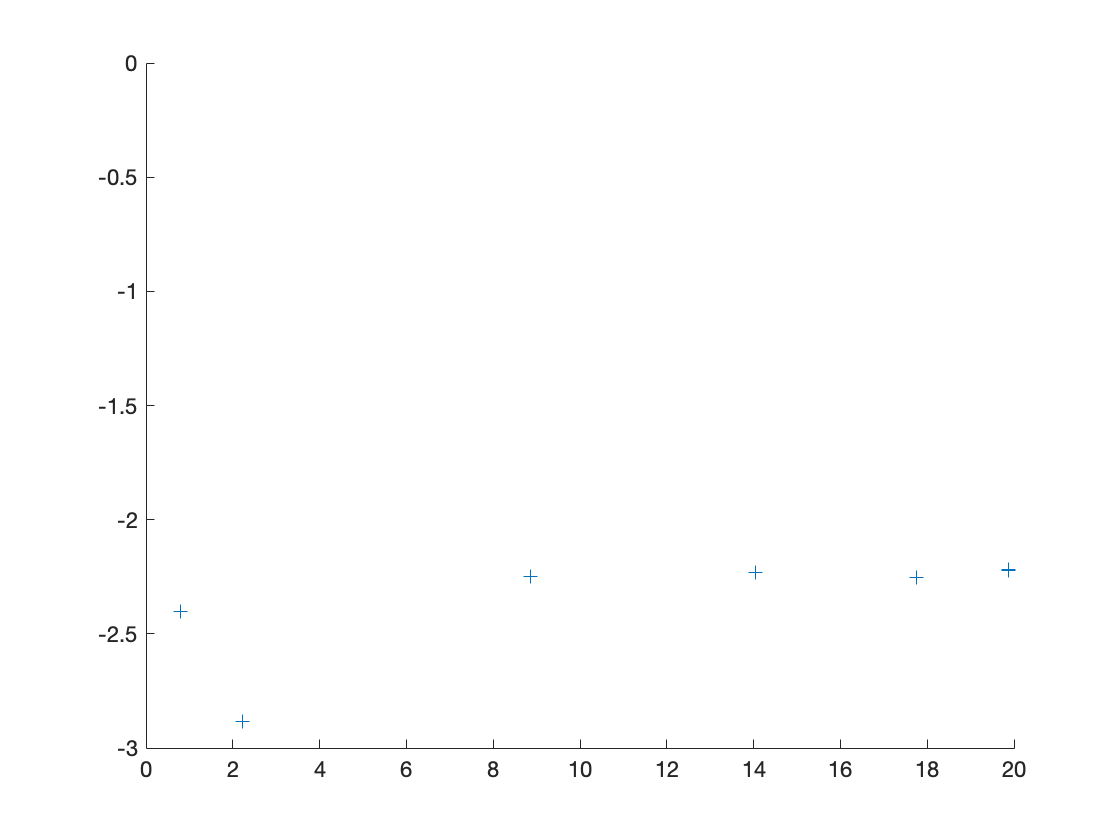}
\caption{The computed poles of the unit square with $n_i = 0.25$ with mesh size $h = h_5$.}
\label{F:ex2_inh0.25}
\end{figure}

\subsection{Example 3}
We consider the L-shaped domain $(-0.5,0.5)^2\setminus [0,0.5]^2$ and set $n_i = 4$, $R = 0.8$ and $N = 20$. For $h = h_5$,
the computed poles in $\Lambda_0$ are listed below
\begin{align*}
&0.8779 - 0.6497i, \quad 2.0162 - 0.4630i, \quad 2.4495 - 0.5704i, \quad 3.4609 - 0.5095i, \\ &0.8382 - 3.3984i, \quad 3.5877 - 0.3881i, \quad 0.8764 - 3.5520i.
\end{align*}
We see that there are no two computed poles that are extremely close to each other, indicating that the orders of these poles are $1$. The results for three poles using  meshes with $h = h_1,\dots,h_5$ are shown in Table~\ref{T:ex3}. Second order of convergence is achieved. 
\begin{table}[!htp]
\centering
\begin{tabular}{c|c|c|c|c|c|c}
\hline
$k_5$ & $0.8779 - 0.6497i$ && $2.0162 - 0.4630i$ && $2.4495 - 0.5704i$ 
\\
\hline
$e_1$ & $-$8.5e-04 $-$ 1.1e-03$i$ & & 6.1e-03 $-$ 9.5e-03$i$ & & 1.2e-02 $-$ 1.8e-02$i$ &
\\
\hline
$e_2$ & $-$2.3e-04 $-$ 2.7e-04$i$ & 1.94 & 1.5e-03 $-$ 2.4e-03$i$ & 1.98 & 3.0e-03 $-$ 4.4e-03$i$ & 2.05
\\
\hline
$e_3$ & $-$6.2e-05 $-$ 7.6e-05$i$ & 1.86 & 4.2e-04 $-$ 6.5e-04$i$ & 1.88 & 8.1e-04 $-$ 1.2e-03$i$ & 1.89
\\
\hline
$e_4$ & $-$1.6e-05 $-$ 1.9e-05$i$ & 1.99 & 1.0e-04 $-$ 1.6e-04$i$ & 2.02 & 2.0e-04 $-$ 2.9e-04$i$ & 2.01
\\
\hline
\end{tabular}
\caption{The errors and convergence order for the L-shaped domain with $n_i = 4$.}
\label{T:ex3}
\end{table}

For $n_i = 0.25$, there are four poles in $\Lambda_0$:
\begin{equation*}
0.8044 - 2.4392i, \quad 0.9602 - 2.9405i, \quad 2.5853 - 3.4012i, \quad 2.7095 - 3.5441i.
\end{equation*}
We show the results in Table~\ref{T:ex3b} for $h = h_1,\dots, h_5$. Again, second order convergence is achieved.
\begin{table}[!htp]
\centering
\begin{tabular}{c|c|c|c|c|c|c}
\hline
$k_5$ & $0.8044 - 2.4392i$ & & $0.9602 - 2.9405i$ & & $2.5853 - 3.4012i$
\\
\hline
$e_1$ & 7.8e-03 $-$ 5.7e-03$i$ & & 7.1e-02 $-$ 4.3e-03$i$ & & 1.1e-01 $+$ 1.0e-01$i$ &
\\
\hline
$e_2$ & 2.2e-03 $-$ 1.5e-03$i$ & 1.84 & 1.8e-02 $-$ 5.1e-03$i$ & 1.95 & 4.6e-02 $+$ 2.7e-02$i$ & 1.46
\\
\hline
$e_3$ & 5.3e-04 $-$ 3.4e-04$i$ & 2.09 & 4.3e-03 $-$ 1.4e-03$i$ & 2.04 & 1.4e-02 $+$ 7.6e-03$i$ & 1.76
\\
\hline
$e_4$ & 1.4e-04 $-$ 8.4e-05$i$ & 1.99 & 1.0e-03 $-$ 3.5e-04$i$ & 2.02 & 3.6e-03 $+$ 1.8e-03$i$ & 1.98
\\
\hline
\end{tabular}
\caption{The errors and
the convergence order for the L-shaped domain with $n_i = 0.25$.}
\label{T:ex3b}
\end{table}
We also compute poles for larger real parts.
The computed poles are shown in Fig.~\ref{F:ex3} for $n = 4$ and Fig.~\ref{F:ex3b} for $n = 0.25$. For $n_i=4$, there are poles with relative larger real parts that are close to real axis. Note that the domain is non-convex. The results suggests that the theory in \cite{Popov1999} also hold for non-convex domains. When $n_i=0.25$, there exists a gap between the poles and the real axis, which is consistent with the \cite{moiola2019acoustic}. 

\begin{figure}[!htp]
\centering
\includegraphics[width=0.6\linewidth]
{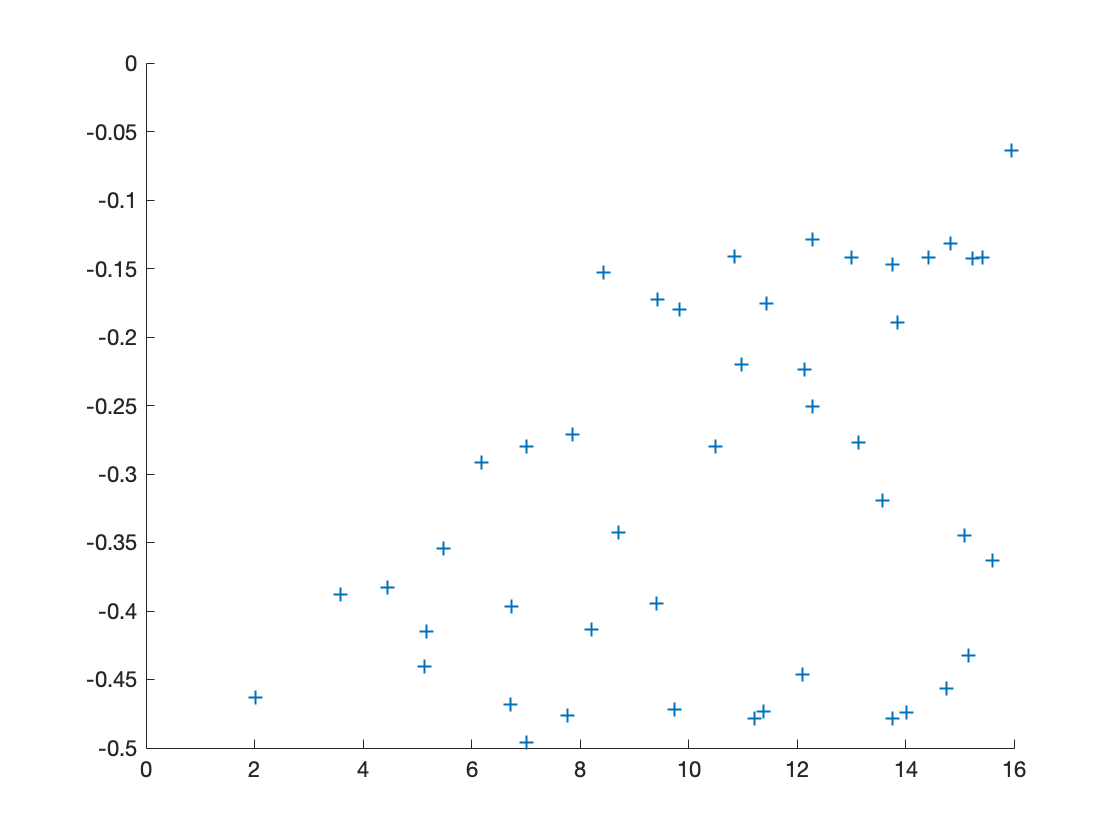}
\caption{The computed poles of the L-shaped domain with $n_i = 4$ and mesh size $h = h_5$.}
\label{F:ex3}
\end{figure}

\begin{figure}[!htp]
\centering
\includegraphics[width=0.6\linewidth]
{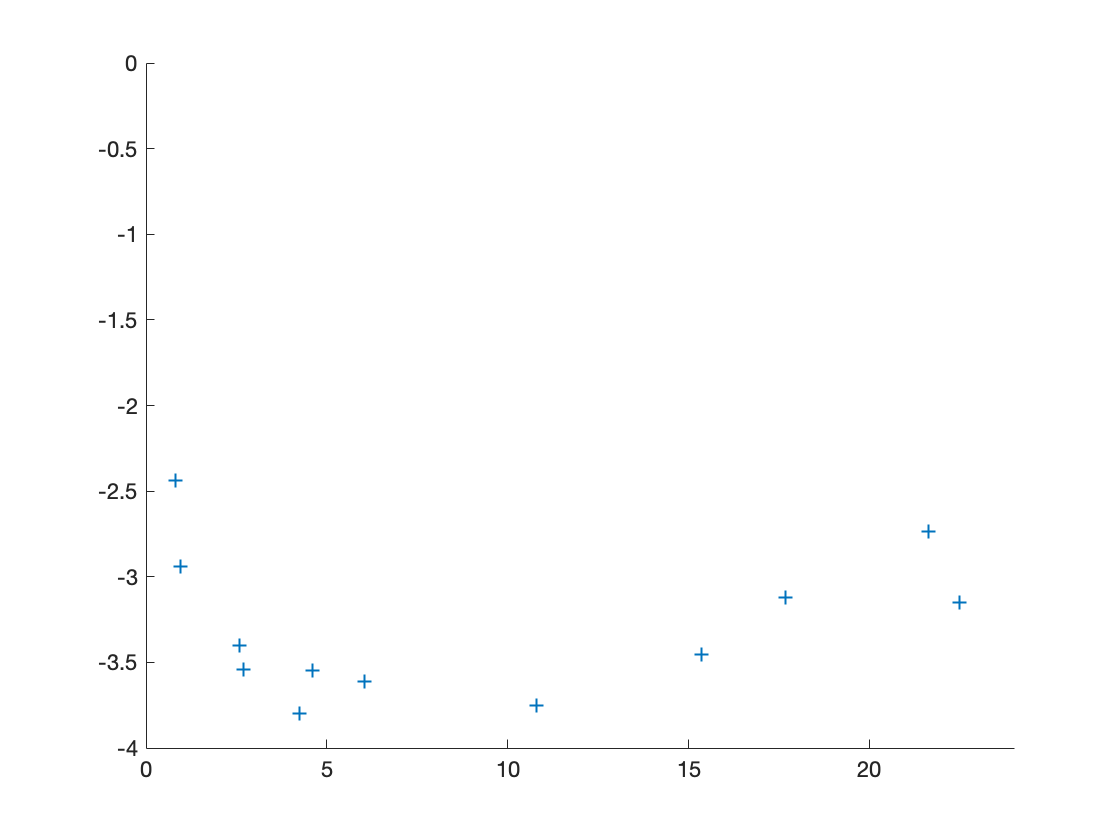}
\caption{The computed poles of the L-shaped domain with $n_i = 0.25$ and mesh size $h = h_6$.}
\label{F:ex3b}
\end{figure}

\subsection{Example 4}
We consider an ellipse with major axis
and minor axis being $2.4$ and $1.6$, respectively. The radius of $\Omega_R$ is set to $R = 1.3$. The computed poles for $n_i = 4$ in $\Lambda_0$ with $h = h_5$ are listed below
\begin{align*}
&0.4439 - 0.3139i, \quad 1.0348 - 0.2279i, \quad 1.2682 - 0.2791i, \quad 1.6496 - 0.1836i, \\& 1.8162 - 0.1896i, \quad 0.4757 - 1.7956i, \quad 0.4505 - 1.8180i, \quad 2.2269 - 0.3009i, \\&2.2678 - 0.1562i, \quad 2.3890 - 0.1363i, \quad 1.3632 - 2.2120i, \quad 1.3710 - 2.2092i.
\end{align*}
As in Example~3, there are no two poles that are extremely close to each other. We show in Table~\ref{T:ex4} the convergence order for three poles.
\begin{table}[!htp]
\centering
\begin{tabular}{c|c|c|c|c|c|c}
\hline
$k_5$ &$0.4439 - 0.3139i$& &$1.0348 - 0.2279i$& &$1.2682 - 0.2791i$&
\\
\hline
$e_1$ & $-$8.2e-05 $-$ 1.3e-04$i$ & & 7.8e-04 $-$ 1.1e-03$i$ & & 1.6e-03 $-$ 2.4e-03$i$ &
\\
\hline
$e_2$ & $-$2.1e-05 $-$ 3.5e-05$i$ & 1.91 & 2.1e-04 $-$ 3.1e-04$i$ & 1.89 & 4.5e-04 $-$ 6.4e-04$i$ & 1.91
\\
\hline
$e_3$ & $-$5.4e-06 $-$ 9.0e-06$i$ & 1.97 & 5.3e-05 $-$ 7.8e-05$i$ & 1.99 & 1.1e-04 $-$ 1.6e-04$i$ & 2.00
\\
\hline
$e_4$ & $-$1.4e-06 $-$ 2.3e-06$i$ & 1.99 & 1.3e-05 $-$ 1.9e-05$i$ & 1.99 & 2.8e-05 $-$ 4.0e-05$i$ & 1.99
\\
\hline
\end{tabular}
\caption{The error and convergence order for the ellipse with $n_i = 4$.}
\label{T:ex4}
\end{table}

For $n_i = 0.25$, the poles in $\Lambda_0$ are
\begin{align*}
&0.4203 - 1.2660i, \quad 0.5122 - 1.5492i, \quad 1.3795 - 1.7090i, \quad 1.3830 - 1.8082i, \\& 0.4540 - 2.7306i, \quad 0.4573 - 2.7448i, \quad 2.2962 - 2.0217i, \quad 2.2899 - 2.0499i, \\& 1.3589 - 3.2323i, \quad 1.3590 - 3.2361i, \quad 3.2105 - 2.2579i, \quad 3.2065 - 2.2655i.
\end{align*}
The results for three poles are shown in Table~\ref{T:ex4b}. Second order convergence is achieved.

\begin{table}[!htp]
\centering
\begin{tabular}{c|c|c|c|c|c|c}
\hline
$k_5$ & $0.4203 - 1.2660i$ & & $0.5122 - 1.5492i$ & & $1.3795 - 1.7090i$
\\
\hline
$e_1$ & 4.1e-04 $-$ 7.3e-04$i$ & & 2.9e-03 $-$ 2.1e-03$i$ & & 4.9e-03 $-$ 7.8e-03$i$ &
\\
\hline
$e_2$ & 1.1e-04 $-$ 1.3e-04$i$ & 2.30 & 6.9e-04 $-$ 4.6e-04$i$ & 2.12 & 1.4e-03 $-$ 1.8e-03$i$ & 2.04
\\
\hline
$e_3$ & 2.6e-05 $-$ 2.5e-05$i$ & 2.25 & 1.7e-04 $-$ 1.1e-04$i$ & 2.03 & 3.4e-04 $-$ 4.0e-04$i$ & 2.10
\\
\hline
$e_4$ & 6.3e-06 $-$ 5.3e-06$i$ & 2.12 & 4.3e-05 $-$ 2.5e-05$i$ & 2.03 & 8.5e-05 $-$ 9.5e-05$i$ & 2.04
\\
\hline
\end{tabular}
\caption{The error and
convergence order for the ellipse with $n_i = 0.25$.}
\label{T:ex4b}
\end{table}

\section{Conclusion}
In this paper, we propose a finite element method for scattering resonances of transmission problems. The resonances are formulated as eigenvalues of a holomorphic Fredholm operator function. The unbounded domain for the transmission problem is truncated and the DtN mapping is employed. Then a finite element method is used to discretize the operator function. Finally, a contour integral based method is used to find the eigenvalues of the resulting nonlinear matrix function. The optimal convergence of the eigenvalues is proved using the abstract approximation theory by Karma. Extensive numerical experiments were carried out to demonstrate the effectiveness of the proposed method and verify the theoretical results in literature. The results on non-convex domains suggests that theory in \cite{Popov1999} also hold for non-convex domains.

The proposed numerical method can be used to study inverse scattering problems related to target deformations, and defects in non-destructive testing. The computational results can also provide guidance for further theoretical investigations.  Extensions of the proposed method to Maxwell’s equations and elastic wave equations will be investigated in future. 

\section*{Acknowledgment}
The research of B. Gong is partially supported by National Natural Science Foundation of China No.12201019. The research of J. Sun is partially supported by an NSF Grant DMS-2109949 and a SIMONS Foundation Collaboration Grant 711922.


\end{document}